\documentclass[12pt, oneside, reqno]{amsart}
\pdfoutput=1
\usepackage{amsfonts}
\usepackage{amssymb}
\usepackage{amsmath}
\usepackage{amsrefs}
\usepackage{tikz-cd}
\usepackage{centernot}
\usepackage{enumerate}
\usepackage{fullpage}
\usepackage{multicol}
\usepackage{amsthm,amsmath,amssymb,tikz,tikz-cd,amsmath,mathrsfs,mathtools,multicol,dirtytalk,tabularx,xy,lipsum,url,enumitem,cmll}
\usepackage{float}

\newtheorem{thm}{Theorem}[section]
\newtheorem{lem}[thm]{Lemma}
\newtheorem{cor}[thm]{Corollary}
\newtheorem{prop}[thm]{Proposition}

\theoremstyle{definition}
\newtheorem{defn}[thm]{Definition}

\theoremstyle{remark}

\newtheorem{ex}[thm]{Example}

\newcommand{\ts}{\textsc}
\newcommand{\sleq}{\sqsubseteq}
\newcommand{\sj}{\sqcup}
\newcommand{\sm}{\sqcap}

\author{John Harding}
\address{New Mexico State University, Las Cruces NM 88003, USA}
\email{jharding@nmsu.edu}

\author{Gejza Jen\v{c}a}
\address{
Department of Mathematics and Descriptive Geometry, Faculty
of Civil Engineering, Slovak University of Technology,
Radlinského 11, Bratislava, 810 05, Slovakia.}
\email{gejza.jenca@stuba.sk}

\author{Bert Lindenhovius}
\address{Institute of Mathematical Methods in Medicine and Data Based Modeling,
Johannes Kepler University Linz, Altenberger Strasse 69 4040 Linz, Austria}
\email{albertus.lindenhovius@jku.at}

\thanks{The second listed author was partially supported by NSF grant DMS-2231414. 
}

\begin{document}

\title{From orthoposets to orthomodular posets}

\begin{abstract}
We show that the category of orthomodular posets is a full coreflective subcategory of the category of strong orthoposets, those orthoposets in which any two orthogonal elements have a join. This coreflection is obtained by building from a strong orthoposet~$P$, an orthomodular poset with the same underlying set and same orthocomplementation as~$P$, but with modified order. This coreflector restricts to a functor from the category of ortholattices to the category of orthomodular posets, and this functor is right adjoint. 
  
\par
\vspace{.2cm}
\noindent \textbf{Mathematics Subject Classification (2020):} 06C15; 81P10; 03G12.       
\par
\vspace{.2cm}
\noindent \textbf{Keywords:} Ortholattice, orthomodular poset, coreflection, adjunction.     \end{abstract}

\maketitle

\section{Introduction}

Every orthocomplemented poset (\ts{op}) is the set-theoretic union of its Boolean subalgebras. Orthomodular posets (\ts{omp}s) are those whose order is the union of the orderings of their Boolean subalgebras. In this note we show that from an \ts{op} $P$ in which any two orthogonal elements have a join, called here a strong orthoposet (\ts{sop}), we can construct an \ts{omp} on the same underlying set as $P$ whose ordering is the union of the orderings of the Boolean subalgebras of $P$. This construction provides a coreflector from the category of \ts{sop}s to its full subcategory of \ts{omp}s that restricts to a right adjoint functor from the category of ortholattices (\ts{ol}s) to \ts{omp}s. 

In the second section we provide preliminaries. In the third section we establish this construction, and in the fourth section we consider the scope of this construction and investigate its basic properties. The final section provides the categorical view. For general background on \ts{ol}s and \ts{omp}s, see \cite{Kalmbach,PtakPulmannova}.

\section{Preliminaries}

Throughout this note, we will be concerned with joins and meets of a subset of a bounded poset $P$. Of course, these do not always exist. For $A\subseteq P$, we will say that $A$ has a join if there is a least upper bound of $A$ in $P$, and similarly for meets. In this case, we write $\bigvee A$ for its join, and $\bigwedge A$ for its meet. When $A$ consists of two elements $x,y$ we use $x\vee y$ and $x\wedge y$ for these. We often write $\bigvee A=x$ or $\bigwedge A = y$ to indicate both  that $A$ has a join or meet, and to also give the value of this join or meet. In particular, if we write $x\vee y =1$, we mean that $x,y$ have a least upper bound in the poset and that the value of this least upper bound is 1. 
 
\begin{defn}
An \emph{orthocomplemented poset} (\ts{op}) is a bounded poset $P$ with a unary operation~$'$ that is order-inverting, period two, and for each $x\in P$ we have $x\vee x'=1$. 
\end{defn}

In an \ts{op} $P$, the operation $'$ is an isomorphism between $P$ and its dual, so we have versions of De Morgan's laws:

\begin{lem}
If $P$ is an orthocomplemented poset, $x\in P$, and $A\subseteq P$, then 
\vspace{1ex} 

\begin{enumerate}
\item if $\bigvee A=x$, then $\bigwedge \{a'\mid a\in A\}=x'$;
\item if $\bigwedge A=x$, then $\bigvee \{a'\mid a\in A\}=x'$. 
\end{enumerate}
\vspace{1ex}

\noindent In particular, $x\wedge x'=0$ for all $x\in P$. 
\end{lem}

Two elements $x,y$ in an orthocomplemented poset are \emph{orthogonal}, written $x\perp y$, if $x\leq y'$. We next define various classes of \ts{op}s that will be considered throughout this note. 

\begin{defn}
Let $P$ be an \ts{op}. We say 
\begin{enumerate}
\item $P$ is a strong orthoposet (\ts{sop}) if any two orthogonal elements have a join;
\item $P$ is an orthomodular poset (\ts{omp}) if it is an \ts{sop} and $x\leq y\Rightarrow y=x\vee(x'\wedge y)$;
\item $P$ is an ortholattice (\ts{ol}) if any two elements have a join;
\item $P$ is an orthomodular lattice (\ts{oml}) if it is an \ts{omp} and an \ts{ol}.
\end{enumerate}
\end{defn}
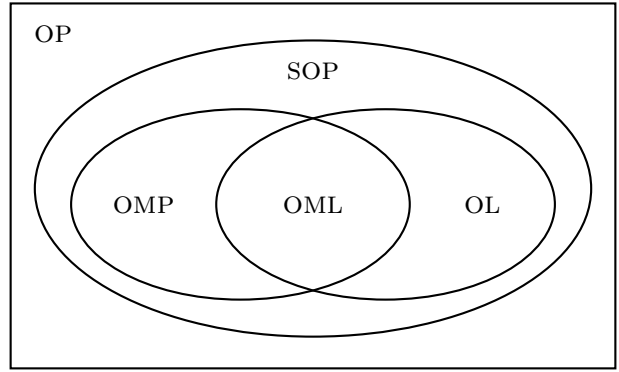
\vspace{1ex}

\begin{center}
\begin{tikzpicture}[thick, font=\sffamily, xscale=.8, yscale=.7]

\draw (0,-.4) rectangle (10,6.5);
\node[anchor=north west] at (0.2,6.3) {\ts{op}};

\draw (5,3) ellipse (4.6cm and 2.8cm);
\node[anchor=north] at (5,5.6) {\ts{sop}};

\draw (3.8,2.7) ellipse (2.8cm and 1.8cm);
\node at (2.2,2.7) {\ts{omp}};

\draw (6.2,2.7) ellipse (2.8cm and 1.8cm);
\node at (7.8,2.7) {\ts{ol}};

\node at (5,2.7) {\ts{oml}};

\end{tikzpicture}
\end{center}
\vspace{1ex}

\begin{defn}
For an \ts{op} $P$, a set $B\subseteq P$ is a Boolean subalgebra of $P$ if $B$ is closed under orthocomplementation and contains $0,1$; any two elements of $B$ have a join and meet in $P$ and these belongs to $B$; and $B$ is a Boolean algebra under the operations inherited from $P$. 
\end{defn}

We develop basic properties of Boolean subalgebras of \ts{op}s.

\begin{lem} \label{lem: tech}
Let $P$ be an \ts{op} and $x\leq y$ elements of $P$. These are equivalent.  
\begin{enumerate}
\item $x,y$ belong to a Boolean subalgebra of $P$;
\item $x,y$ belong to a Boolean subalgebra of $P$ with at most 8 elements;
\item $y=x\vee(x'\wedge y)$ and $x'=y'\vee(x'\wedge y)$.
\end{enumerate}
\vspace{1ex}

\noindent In the final condition, it is required that the indicated joins and meets in the expressions exist and that the indicated equalities hold. 
\end{lem}

\begin{proof}
Since a Boolean algebra generated by two comparable elements has at most 8 elements, (1) $\Leftrightarrow$ (2). (1) $\Rightarrow$ (3) since the conditions hold in any Boolean algebra, and in a Boolean subalgebra of $P$, all binary joins and meets and the orthocomplements are those of $P$. For (3) $\Rightarrow$ (2), let $B=\{0,x,y',x'\wedge y,x',y,x\vee y',1\}$. Note that the existence of $x'\wedge y$ is given by the third condition, and hence by De Morgan's law so is the existence of $x\vee y'$. The set $B$ has at most 8 elements, contains $0,1$ and by DeMorgan's laws is closed under orthocomplementation. With three exceptions, every pair of elements from $B$ is of the form $\{u,v\}$ where $u\leq v$ or $u'\leq v$, and any such pair has a join in $P$ and that join belongs to $B$. The three exceptions are $\{x,y'\}$, $\{x,x'\wedge y\}$, and $\{y',x'\wedge y\}$. Our assumptions imply their join exists in $P$ and belongs to $B$. Then all joins and meets in $B$ are as indicated in the diagram below, and it follows that $B$ is a homomorphic image of an 8-element Boolean algebra.  
\end{proof}

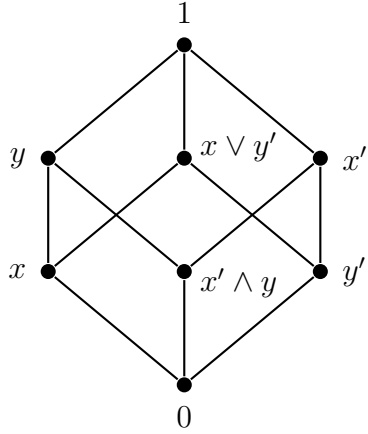
\begin{center}
\begin{tikzpicture}[thick, font=\sffamily, x=1.8cm, y=1.5cm, dot/.style={circle, fill=black, inner sep=2pt}]

\node[dot] (0) at (0,0) {};
\node[below=4pt] at (0) {$0$};

\node[dot] (x) at (-1,1) {};
\node[left=4pt] at (x) {$x$};
    
\node[dot] (a1) at (0,1) {};
\node at (0.4,.9) {$x'\wedge y$};
    
\node[dot] (a2) at (1,1) {};
\node[right=4pt] at (a2) {$y'$};

\node[dot] (y) at (-1,2) {};
\node[left=4pt] at (y) {$y$};
    
\node[dot] (c1) at (0,2) {};
\node at (0.4,2.1) {$x\vee y'$};
    
\node[dot] (c2) at (1,2) {};
\node[right=4pt] at (c2) {$x'$};

\node[dot] (1) at (0,3) {};
\node[above=4pt] at (1) {$1$};

\draw (0) -- (x); \draw (0) -- (a1); \draw (0) -- (a2);
\draw (x) -- (y); \draw (x) -- (c1);
\draw (a1) -- (y); \draw (a1) -- (c2);
\draw (a2) -- (c1); \draw (a2) -- (c2);
\draw (y) -- (1); \draw (c1) -- (1); \draw (c2) -- (1);
\end{tikzpicture}
\end{center}

\begin{cor}
\label{cor: char OMP}
For an \ts{op} $P$, these are equivalent
\vspace{1ex}

\begin{enumerate}
\item $P$ is an \ts{omp};
\item If $x,y\in P$ and $x\leq y$, then $x,y$ belong to a Boolean subalgebra of $P$.
\end{enumerate}
\vspace{1ex}

\noindent Thus, an \ts{op} is an \ts{omp} iff its ordering is the union of the orderings of its Boolean subalgebras. 
\end{cor}

\begin{proof}
(1) $\Rightarrow$ (2) If $x\leq y$ in an \ts{omp}, then $x\vee(x'\wedge y)=y$ by the definition of an \ts{omp}. But $y'\leq x'$, so $x'=y'\vee(x'\wedge y)$ by the definition of an \ts{omp}. So by Lemma~\ref{lem: tech}, $x,y$ lie in a Boolean subalgebra of $B$. (2) $\Rightarrow$ (1) If $x\leq y$, then $x\vee (x'\wedge y)=y$ since they lie in a Boolean subalgebra.     
\end{proof}

\section{The main construction}

\begin{defn}
For $P=(P,\leq,\,',0,1)$ an \ts{op}, let $G(P)=(P,\sleq,\,',0,1)$ where $x \sleq y$ iff $x\leq y$ and $x,y$ belong to a Boolean subalgebra of $P$.
\end{defn}

\begin{thm} \label{thm: main}
If $P$ is a strong orthoposet, then $G(P)$ is an orthomodular poset.
\end{thm}

\begin{proof}
Clearly $\sleq$ is reflexive and anti-symmetric, and it follows from the definition that $x\sleq y$ iff $y'\sleq x'$. To show it is transitive, we first establish the following claim. 
\vspace{1ex}

\noindent {\bf Claim 1:} If $x\sleq y$ and $y\sleq z$, then $x\vee(x'\wedge z)=z$.

\begin{proof}[Proof of claim]
Since $P$ is an \ts{sop}, the orthogonal join $x\vee z'$ exists, so by De Morgan's law $x'\wedge z$ exists, and therefore the orthogonal join $x\vee(x'\wedge z)$ exists. Surely $x\vee (x'\wedge z)\leq z$. For the other inequality, note $x'\wedge y$ and $y'\wedge z$ are orthogonal, so their join $(x\wedge y')\vee(y'\wedge z)$ exists, and since $y\leq z$ and $y'\leq x'$ we have $(x'\wedge y)\vee(y'\wedge z)\leq x'\wedge z$. Therefore we have $x\vee[(x'\wedge y)\vee(y'\wedge z)]\leq x\vee(x'\wedge z)$, noting that these joins exist since they are orthogonal joins. Then since $x\vee(x'\wedge y)=y$ we have $y\vee(y'\wedge z)\leq x\vee(x'\wedge z)$, and then since $y\vee(y'\wedge z)=z$, we have $z\leq x\vee(x'\wedge z)$. Thus $x\vee(x'\wedge z)=z$.
\end{proof}

\noindent {\bf Claim 2:} $x\sleq y$ iff $y'\sleq x'$.

\begin{proof}[Proof of claim] 
This follows directly from the definition of $\sleq$ and the fact that $'$ is an orthocomplementation. 
\end{proof}

\noindent {\bf Claim 3:} $x\sleq y$ and $y\sleq z$ imply $x\sleq z$.

\begin{proof}[Proof of claim]
By Claim~1 we have $x\vee(x'\wedge y)=z$. But Claim~2 gives $z'\sleq y'$ and $y'\sleq x'$. Applying Claim~1 to this gives $z'\vee(x'\wedge z)=x'$. So by Lemma~\ref{lem: tech}, we have that $x,z$ lie in a Boolean subalgebra of $P$. Clearly $x\leq z$, so $x\sleq z$. 
\end{proof} 

We have now shown that $\sleq$ is a partial ordering. Since $\{0,x,x',1\}$ is a Boolean subalgebra of $P$ for any $x\in P$, it follows that $0\sleq x$ and $x\sleq 1$ for any $x\in P$. So $(P,\sleq)$ is a bounded poset. Clearly $'$ is period two, and Claim~2 gives that it is order-inverting. Since it is more difficult to be related by $\sleq$ than by $\leq$, we have that $x$ and $x'$ are complements in $(P,\sleq)$. Thus $(P,\sleq,')$ is an \ts{op}. 
\vspace{1ex}

We use $\sm$ and $\sj$ to denote existing binary meets and joins in $(P,\sleq,\,',0,1)$. As usual, we use $x\sm y=z$ and $x\sj y=z$ to indicate both that the given meet or join exists and is equal to the expressed element. 
\vspace{1ex}

\noindent {\bf Claim~4:} If $x,y$ lie in a Boolean subalgebra of $P$, then $x\sm y=x\wedge y$ and $x\sj y=x\vee y$. 

\begin{proof}[Proof of claim]
Suppose $x,y$ are in the Boolean subalgebra $B$ of $P$. Then the join $x\vee y$ exists in $P$ and is their join in $B$. But then $x,y\sleq x\vee y$. Since any upper bound of $x,y$ in $(P,\sleq)$ is also an upper bound in $(P,\leq)$, $x\vee y$ is the least upper bound of $x,y$ in $(P,\sleq)$. Similar comments hold for meets.   
\end{proof}

To show that $G(P)$ is an \ts{omp}, let $x\sleq y$. Then $\{0,x,y,x\vee y',x',y',x'\wedge y,1\}$ is a Boolean subalgebra of $P$. Using Claim~4 repeatedly, $x\sj(x'\sm y)=x\sj(x'\wedge y)=x\vee(x'\wedge y)=y$. This establishes orthomodularity, and concludes the proof of the theorem.  
\end{proof}

\begin{cor} \label{cor: OL}
If $L$ is an \ts{ol}, then $G(L)$ is an \ts{omp}.    
\end{cor}

\section{Properties and limitations of the main construction}

Here we collect several properties of this construction and several examples to show the limitations of its scope. 

\begin{prop}
\label{prop: same Boolean subalgebras}
If $P$ is a \ts{sop}, then $P$ and $G(P)$ have the same Boolean subalgebras, and for elements $x,y$ in a Boolean subalgebra $B$, their join and meet in $B$ agrees with their join and meet in $P$ and with their join and meet in $G(P)$.      
\end{prop}

\begin{proof}
Claim~4 in the proof of Theorem~\ref{thm: main} shows that if $B$ is a Boolean subalgebra of $P$, then $B$ is a Boolean subalgebra of $G(P)$. For the converse, we first show 
\vspace{2ex}

\noindent {\bf Claim:} If $x,y$ lie in a Boolean subalgebra of $G(P)$, then $x\wedge y=x\sm y$ and $x\vee y=x\sj y$. 

\begin{proof}[Proof of claim]
Note first that if $p,q$ are orthogonal in $G(P)$, then they are orthogonal in $P$ and since $p\sleq q'$ we have that $p,q$ belong to a Boolean subalgebra of $P$, hence by Claim~4 of Theorem~\ref{thm: main}, $p\vee q=p\sqcup q$. Suppose $x,y$ belong to a Boolean subalgebra of $G(P)$. Then there are $p,q,r$ in $G(P)$ that are pairwise orthogonal and with $p\sj q = x$ and $q\sj r = y$. But by the remark above, $p,q,r$ are also pairwise orthogonal in $P$ and $p\vee q=x$ and $q\vee r= y$. Then $x,r$ are orthogonal in both $G(P)$ and $P$, giving that $x\sj r=x\vee r$ in the sense that both joins exist and are equal. But it is easily seen that $x\sj r=x\sj y$ and $x\vee r = x\vee y$. 
\end{proof}

To conclude the proof, suppose $D$ is a Boolean subalgebra of $G(P)$. Surely $D$ contains $0,1$ and is closed under the common orthocomplementation of $P$ and $G(P)$. To see that $D$ is a Boolean subalgebra of $P$ it is enough to show that if $x,y\in D$, then $x\wedge y$ and $x\vee y$ exist in $P$, and that for $x,y,z\in D$ that $x\wedge(y\vee z)=(x\wedge y)\vee(x\wedge z)$. All of this follows from applications of the claim.  
\end{proof}

The following is an immediate consequence of Corollary~\ref{cor: char OMP}.

\begin{prop}
\label{prop: P OMP iff G(P)=P}
Let $P$ be an \ts{sop}. Then $P$ is an \ts{omp} iff $P=G(P)$.     
\end{prop}

We next provide an example to show that the Main Construction is not well-behaved when applied to general \ts{op}s rather than \ts{sop}s. The assumption of the existence of binary orthogonal joins is necessary to obtain transitivity. 

\begin{ex}
Consider the \ts{op} shown below.     

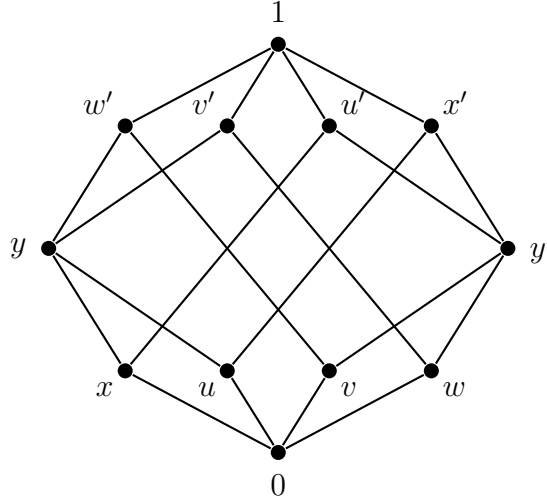
\begin{center}
\begin{tikzpicture}[thick, font=\sffamily, x=1.5cm, y=1.2cm,
    dot/.style={circle, fill=black, inner sep=2pt}, xscale=.9, yscale=.9]

    \node[dot] (0) at (0,0) {};
    \node[below=4pt] at (0) {$0$};

    \node[dot] (x) at (-1.5, 1) {};
    \node[below left] at (x) {$x$};
    
    \node[dot] (u) at (-0.5, 1) {};
    \node[below left] at (u) {$u$};
    
    \node[dot] (v) at (0.5, 1) {};
    \node[below right] at (v) {$v$};
    
    \node[dot] (w) at (1.5, 1) {};
    \node[below right] at (w) {$w$};

    \node[dot] (y) at (-2.25, 2.5) {};
    \node[left=4pt] at (y) {$y$};
    
    \node[dot] (yprime) at (2.25, 2.5) {};
    \node[right=4pt] at (yprime) {$y'$};

    \node[dot] (wprime) at (-1.5, 4) {};
    \node[above left] at (wprime) {$w'$};
    
    \node[dot] (vprime) at (-0.5, 4) {};
    \node[above left] at (vprime) {$v'$};
    
    \node[dot] (uprime) at (0.5, 4) {};
    \node[above right] at (uprime) {$u'$};
    
    \node[dot] (xprime) at (1.5, 4) {};
    \node[above right] at (xprime) {$x'$};

    \node[dot] (1) at (0, 5) {};
    \node[above=4pt] at (1) {$1$};

    \draw (0) -- (x);
    \draw (0) -- (u);
    \draw (0) -- (v);
    \draw (0) -- (w);

    \draw (x) -- (y);
    \draw (x) -- (uprime);
    
    \draw (u) -- (y);
    \draw (u) -- (xprime);
    
    \draw (v) -- (wprime);
    \draw (v) -- (yprime);
    
    \draw (w) -- (vprime);
    \draw (w) -- (yprime);

    \draw (y) -- (wprime);
    \draw (y) -- (vprime);
    
    \draw (yprime) -- (uprime);
    \draw (yprime) -- (xprime);

    \draw (wprime) -- (1);
    \draw (vprime) -- (1);
    \draw (uprime) -- (1);
    \draw (xprime) -- (1);
\end{tikzpicture}
\end{center}
\vspace{1ex}

\noindent In this \ts{op} we have $x'\wedge y=u$ and $w'\wedge y'=v$. One can see that $x\sleq y$ and $y\sleq w'$. However we do not have $x\sleq w'$ since $w'\wedge x'$ does not exist, both $u,v$ are lower bounds of $w',v'$. This \ts{op} is not a \ts{sop}. 
\end{ex}

Finally, by Corollary~\ref{cor: OL}, if $L$ is an \ts{ol}, then $G(L)$ is an \ts{omp}. We might hope that it is further the case that $G(L)$ is an \ts{oml}. This is not the case. 

\begin{ex}
Consider the \ts{omp} $P$ known as the 4-loop, whose Greechie diagram consists of four 8-element Boolean algebras glued together as shown below. This \ts{omp} is not a lattice, the elements $a,e$ have both $c',g'$ as upper bounds, so they have no least upper bound, and similarly $c,g$ have no least upper bound. 
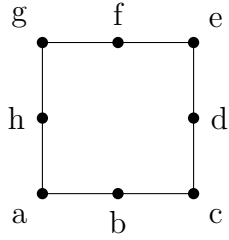
\vspace{1ex}

\begin{center}
\begin{tikzpicture}

  \coordinate (A) at (0,0);
  \coordinate (B) at (1,0);
  \coordinate (C) at (2,0);
  \coordinate (D) at (2,1);
  \coordinate (E) at (2,2);
  \coordinate (F) at (1,2);
  \coordinate (G) at (0,2);
  \coordinate (H) at (0,1);

  \draw (A) -- (C) -- (E) -- (G) -- cycle;
  
  \node[circle, fill, inner sep=1.5pt, label={below left:a}] at (A) {};
  \node[circle, fill, inner sep=1.5pt, label={below:b}] at (B) {};
  \node[circle, fill, inner sep=1.5pt, label={below right:c}] at (C) {};
  
  \node[circle, fill, inner sep=1.5pt, label={right:d}] at (D) {};
  
  \node[circle, fill, inner sep=1.5pt, label={above right:e}] at (E) {};
  \node[circle, fill, inner sep=1.5pt, label={above:f}] at (F) {};
  \node[circle, fill, inner sep=1.5pt, label={above left:g}] at (G) {};
  
  \node[circle, fill, inner sep=1.5pt, label={left:h}] at (H) {};

\end{tikzpicture}
\end{center}
\vspace{1ex}

\noindent Let $\overline{P}$ be the MacNeille completion of $P$. Then $\overline{P}$ is an \ts{ol}. It is formed by adding two new elements to $P$, one element $y$ that is the join of $a,e$ and the meet of $c',g'$, and its orthocomplement $y'$ that is the join of $c,g$ and the meet of $a',e'$. This is not an \ts{oml}. Since $\overline{P}$ is an \ts{ol} we can apply our construction to it and form the \ts{omp} $G(\overline{P})$. Since MacNeille completions preserve all existing joins and meets, the restriction of the ordering of $G(\overline{P})$ to its subset $P$ is the original ordering of $P$. One can check that in $G(\overline{P})$ there are no comparabilities involving $y$ or $y'$ other than the trivial ones with $0$ and $1$. For instance $a\not\leq y$ since in $\overline{P}$ we have $a'\wedge y = a'\wedge c'\wedge g'=0$. Thus $G(\overline{P})$ is the horizontal sum of $P$ and a 4-element Boolean algebra, and this is not a lattice. In this case, applying $G$ to $\overline{P}$ essentially removes the joins we attempted to include using the MacNeille completion.   
\end{ex}

\section{A categorical view}

\begin{defn}
A map $f:P\to Q$ between \ts{sop}s is an \ts{sop}-morphism if 
\vspace{1ex}

\begin{enumerate}
\item $f$ is order-preserving;
\item $f(x')=f(x)'$;
\item $x\perp y\Rightarrow f(x\vee y)=f(x)\vee f(y)$.
\end{enumerate}
\end{defn}

In this definition, if $x\perp y$, then $x\leq y'$, so $f(x)\leq f(y')=f(x)$. Thus an \ts{sop} morphism is an order-preserving map that preserves orthocomplementation and binary orthogonal joins. 

\begin{lem}
\label{lem: G(f)}
Let $P,Q$ be \ts{sop}s and $f:P\to Q$ be an \ts{sop}-morphism. Then the same set map $f$ considered as a map $f:G(P)\to G(Q)$ is an \ts{sop}-morphism.     
\end{lem}

\begin{proof}
We use the convention that $\leq$ represents an ordering in $P$ or $Q$, which one is clear from the context, and $\sleq$ represents an ordering in $G(P)$ or $G(Q)$. 
\vspace{2ex}

\noindent {\bf Claim:} If $x\sleq y$ in $G(P)$, then $f(x)\sleq f(y)$ in $G(Q)$.

\begin{proof}[Proof of claim]
Let $x\sleq y$. Then $x\leq y$ and $x,y$ lie in a Boolean subalgebra of $P$. So we have that $x\vee (x'\wedge y)=y$ and $y'\vee(x'\wedge y)=x'$. Since $x,y'$ are orthogonal in $P$ and $f$ is an \ts{sop}-morphism, we have $f(x),f(y)'$ are orthogonal in $Q$ and $f(x\vee y')=f(x)\vee f(y)'$ in $Q$. Thus $f(x'\wedge y)=f(x)'\wedge f(y)$. Then using again that $f$ preserves orthogonal joins between $P$ and $Q$ we have $f(x)\vee f(x'\wedge y)=f(y)$, hence $f(x)\vee(f(x)'\wedge f(y))=f(y)$, and similarly $f(y)'\vee(f(x)'\wedge f(y))=f(x)'$. So by Lemma~\ref{lem: tech} we have $f(x)$ and $f(y)$ belong to a Boolean subalgebra of $Q$, giving $f(x)\sleq f(y)$.  
\end{proof}

If $x,y$ are orthogonal in $G(P)$, then $x\sleq y'$, so $f(x)\sleq f(y)'$, and hence $f(x)$ and $f(y)$ are orthogonal in $G(Q)$. Using Proposition~\ref{prop: same Boolean subalgebras}, we have that join of $x,y$ in $G(P)$ is given by their join $x\vee y$ in $P$, and the join of $f(x),f(y)$ in $G(Q)$ is given by their join $f(x)\vee f(y)$ in $Q$. Since $f$ is an \ts{sop}-morphism and $x\vee y$ is an orthogonal join, we have $f(x\vee y)=f(x)\vee f(y)$, hence $f(x\sj y)=f(x\vee y)=f(x)\vee f(y)=f(x)\sj f(y)$. Thus $f:G(P)\to G(Q)$ is an \ts{sop}-morphism. 
\end{proof}


Clearly the identity map is an \ts{sop}-morphism and the composite of \ts{sop}-morphisms is again an \ts{sop}-morphism. So we may define  

\begin{defn}
Let \textsf{SOP} be the category of strong orthoposets and the \ts{sop}-morphisms between them and let its full subcategory consisting of orthomodular posets be  \textsf{OMP}.
\end{defn}

It follows from Theorem~\ref{thm: main} and Lemma~\ref{lem: G(f)} that there is a functor $G:\textsf{SOP}\to\textsf{OMP}$ taking an \ts{sop} $P$ to the \ts{oml} $G(P)$ provided by our main construction and taking an \ts{sop}-morphism $f:P\to Q$ to the \ts{sop}-morphism $G(f):G(P)\to G(Q)$ where $G(f)$ is the same set mapping as $f$. Further, by Proposition~\ref{prop: P OMP iff G(P)=P} $G$ restricts to the identity functor on \textsf{OMP}. 

\begin{thm}
\emph{\textsf{OMP}} is a full coreflective subcategory of \emph{\textsf{SOP}} with coreflector $G$.    
\end{thm}

\begin{proof}
We must show that the functor $G:\textsf{SOP}\to\textsf{OMP}$ is right adjoint to the inclusion functor $\iota:\textsf{OMP}\hookrightarrow\textsf{SOP}$. For this, it is enough to show that for each \ts{omp} $Q$, there is an \ts{sop}-morphism $\varepsilon_Q:G(Q)\to Q$ such that for any \ts{omp} $P$ and any \ts{sop}-morphism $f:P\to Q$, there is a unique \ts{sop}-morphism $f':P\to G(Q)$ with $\varepsilon_Q\circ f'=f$. 

\begin{center}
\begin{tikzpicture}[node distance=3cm, auto, thick, >={Stealth[length=3mm]}]
    
    \node (P) at (0, 2.5) {$P$};
    \node (Q) at (3.5, 2.5) {$Q$};
    \node (GQ) at (3.5, 0) {$G(Q)$};

    \draw[->] (P) -- (Q) node[midway, above] {$f$};

    \draw[->, dashed] (P) -- (GQ) node[midway, below left] {$f'$};

    \draw[->] (GQ) -- (Q) node[midway, right] {$\varepsilon_Q$};
\end{tikzpicture}
\end{center}
\vspace{1ex}

\noindent Let $\varepsilon_Q$ be the identity map on $Q$. Clearly $\varepsilon_Q$ preserves orthocomplementation and order. If $u,v$ are orthogonal in $G(Q)$, then they belong to a Boolean subalgebra $B$ of $Q$ and by Proposition~\ref{prop: same Boolean subalgebras} the join of $u,v$ in $G(Q)$ agrees with their join in $Q$. This shows that $\varepsilon_Q$ is an \ts{sop}-morphism. Since $P$ is an \ts{omp} we have by Proposition~\ref{prop: P OMP iff G(P)=P} that $G(P)=P$. Since $G$ is a functor, $G(f):G(P)\to G(Q)$ is an \ts{sop}-morphism. So the set mapping $f$ is an \ts{sop}-morphism from $P\to G(Q)$ and is the unique one whose composite with $\varepsilon_Q$ is $f$.
\end{proof}

There is a somewhat different direction that we can consider our construction. Let \textsf{OL} be the category of ortholattices and ortholattice homomorphisms between them. This is a non-full subcategory of \textsf{SOP}, so $G$ restricts to a function $G:\textsf{OL}\to\textsf{OMP}$. 

\begin{thm}
The functor $G:\emph{\textsf{OL}}\to\emph{\textsf{OMP}}$ is right adjoint. 
\end{thm}

\begin{proof}
We use the adjoint functor theorem. Since \textsf{OL} is a variety, it is complete as a category. To see that $G$ preserves small limits, it is enough to see the following. 
\vspace{2ex}

\noindent {\bf Claim:} $G$ preserves products and equalizers. 

\begin{proof}[Proof of claim]
Suppose $L=\prod_IL_i$. As sets, $G(L)=\prod_IG(L_i)$ and orthocomplementation on these agree. For choice functions $(x_i)$ and $(y_i)$ in $L$, we have $(x_i)\sleq(y_i)$ iff $(x_i)\leq(y_i)$ and $(x_i),(y_i)$ generate a Boolean subalgebra. But this occurs iff $x_i\leq y_i$ and $x_i,y_i$ generate a Boolean subalgebra for each $i\in I$, which occurs iff $x_i\sleq y_i$ for each $i\in I$. So $G$ preserves products. Suppose $f,g:L\to M$ are ortholattice homomorphisms. Their equalizer is the subalgebra $S=\{x\mid f(x)=g(x)\}$ of $L$. It is easy to see that the ordering of $G(S)$ is that inherited from $G(L)$, so $G(S)$ is the equalizer of $G(f),G(g):G(L)\to G(M)$.
\end{proof}

Finally, $G$ trivially satisfies the solution set condition. If $P$ is an \ts{omp}, then the idenity map $\iota:P\to G(P)$ is an isomorphism. For any \ts{sop} $Q$ and \ts{omp}-morphism $f:P\to G(Q)$, clearly $f$ is itself a map from $G(P)\to Q$. So the solution set condition is satisfied, hence by the adjoint functor theorem, $G$ is a right adjoint functor. 
\end{proof}


\begin{thebibliography}{99}
\setlength{\parskip}{0em}

\bibitem{Kalmbach}
G.~Kalmbach, \textit{Orthomodular Lattices}, London Mathematical Society Monographs, Vol.~18, Academic Press, London, 1983.

\bibitem{PtakPulmannova}
P.~Pt\'ak and S.~Pulmannov\'a, \textit{Orthomodular Structures as Quantum Logics: Intrinsic Properties, State Space and Probabilistic Topics}, Fundamental Theories of Physics, Vol.~44, Kluwer Academic Publishers, Dordrecht, 1991.


\end{thebibliography}
\end{document}